\documentclass[a4paper,12pt]{article}
\usepackage[authoryear,round]{natbib}
\usepackage{amssymb,amsmath,amsthm}
\usepackage{amssymb,mathrsfs}
\usepackage{amsmath,amssymb,latexsym,amsthm,color,times,fullpage}
\usepackage[pdftex]{hyperref}
\hypersetup{plainpages=True, pdfstartview=FitV,
colorlinks=true,linkcolor=blue,citecolor=blue}
\usepackage{enumitem}

\def\tr#1{\left\lfloor#1\right\rfloor}

\def\dd#1{\,\mbox{d}#1}
\def\le{\leqslant}
\def\ge{\geqslant}

\newtheorem{thm}{Theorem}[section]
\newtheorem{cor}[thm]{Corollary}
\newtheorem{lem}[thm]{Lemma}

\theoremstyle{remark}
\newtheorem{rem}[thm]{Remark}
\title{Limit laws of the coefficients of polynomials
with only unit roots}
\author{Hsien-Kuei Hwang\thanks{Part of the work of this author
was done while visiting ISM (Institute of Statistical Mathematics),
Tokyo; he thanks ISM for its hospitality and support.} \\
    Institute of Statistical Science, \\
    Institute of Information Science\\
    Academia Sinica\\
    Taipei 115\\
    Taiwan
\and Vytas Zacharovas\thanks{Part of the work of this author was done during 
his visit at Institute of Statistical Science, Academia Sinica,
Taipei. He gratefully acknowledges the support 
provided by the institute.} \\
    Dept. Mathematics \&\ Informatics\\
    Vilnius University\\
    Naugarduko 24, Vilnius\\
    Lithuania}
\date{\today}

\begin{document}
\maketitle

\begin{abstract}
We consider sequences of random variables whose probability
generating functions are polynomials all of whose roots lie on the 
unit circle. The distribution of such random variables has only 
been sporadically studied in the literature. We show that the
random variables are asymptotically normally distributed if and only
if the fourth normalized (by the standard deviation) central 
moment tends to $3$, in contrast to the common scenario for
polynomials with only real roots for which a central limit theorem
holds if and only if the variance goes unbounded. We also derive a
representation theorem for all possible limit laws and apply our
results to many concrete examples in the literature, ranging from
combinatorial structures to numerical analysis, and from probability
to analysis of algorithms.
\end{abstract}

\section{Introduction}

The close connection between the location of the zeros of a function
(or a polynomial) and the distribution of its coefficients has long
been the subject of extensive study; typical examples include the
order of an entire function and its zeros in Analysis, and the limit
distribution of the coefficients of polynomials when all roots are
real in Combinatorics, Probability and Statistical Physics. We
address in this paper the situation when the roots of the sequence
of probability generating functions all lie on the unit circle.
While one may convert the situation with only unimodular zeros to
that with only real zeros by a suitable change of variables, such
\emph{root-unitary polynomials} turn out to have many fascinating 
properties due mainly to the boundedness of all zeros. In particular, 
we show that the fourth normalized central moments are (asymptotically) 
always bounded between $1$ and $3$, the limit distribution being 
Bernoulli if they tend to $1$ and Gaussian if they tend to $3$.

Although this class of polynomials does not have a standard name, we
will refer to them as, following \cite{kedlaya08a} and for
convention, root-unitary polynomials. Other related terms include
\emph{self-inversive} (zeros symmetric in the unit circle),
\emph{reciprocal} or \emph{self-reciprocal} ($P(z) = z^n
P(z^{-1})$), \emph{uni-modular} (all coefficients of modulus one),
\emph{palindromic} ($a_j=a_{n-j}$), etc., when $P(z) = \sum_{0\le
j\le n}a_j z^j$ is a polynomial of degree $n$.

Unit roots of polynomials play a very special and important role in
many scientific and engineering disciplines, notably in statistics
and signal processing where the unit root test decides if a time
series variable is non-stationary. On the other hand, many
nonparametric statistics are closely connected to partitions of
integers, which lead to generating functions whose roots all lie on
the unit circle. We will discuss many examples in Sections
\ref{sec-app-i} and \ref{sec-app-ii}. Another famous example is the
Lee-Yang partition function for Ising model, which has stimulated a
widespread interest in the statistical-physical literature since the
1950's.


While there is a large literature on polynomials with only real
zeros, the distribution of the coefficients of root-unitary
polynomials has only been sporadically studied; more references will
be given below. It is well known that for polynomials with
nonnegative coefficients whose roots are all real, one can decompose
the polynomials into products of linear factors, implying that the
associated random variables are expressible as sums of independent
Bernoulli random variables. Thus one obtains \emph{a Gaussian limit
law for the coefficients if and only if the variance tends to infinity}; 
see the survey paper \cite{pitman97a} for more information and for finer
estimates. A representative example is the Stirling numbers of the
second kind for which \cite{harper67a} showed that the generating
polynomials have only real roots\footnote{The fact that the Stirling
polynomials (of the second kind) have only real roots had been known
long before Harper; see for example \cite{docagne87a}; in addition,
\cite{bell38a} wrote (without providing reference) that all results
of d'Ocagne's paper were already obtained thirty years before him by
a number of English authors.}; he also established the asymptotic
normality of these numbers by proving that the variance tends to
infinity. For more examples and information on polynomials with only
real roots, see \cite{brenti94a}, \cite{pitman97a} and the
references therein. See also \cite{haigh71a}, \cite{hayman56a},
\cite{renyi72a} for different extensions.

Our first main result states that if we restrict the range where the
roots of the polynomials $P_n(z)$ can occur to the unit circle
$|z|=1$, then the asymptotic normality of $X_n$ defined by the
coefficients is determined by the limiting behavior of its fourth
normalized central moment. \emph{Throughout this paper, write}
$X_n^* := (X_n - \mathbb{E}(X_{n}))/ \sqrt{\mathbb{V}(X_n)}$.

\begin{thm}\label{limit-law}
Let $\{X_{n}\}$ be a sequence of discrete random variables whose
probability generating functions $\mathbb{E}(z^{X_n})$ are polynomials
of degree $n$ with all roots $\rho_j$ lying on the unit circle
$|\rho_j|=1$.
\begin{description}
\item [--] (Bounds for the fourth normalized central moment)
\begin{align} \label{4th-mm-bds}
    1\le \mathbb{E}\left(X_n^*\right)^4 < 3
    \qquad(n\ge1).
\end{align}
\item [--] (Asymptotic normality)
The sequence of random variables $\{X_n^*\}$ converges in
distribution (and with all moments) to the standard normal law
$\mathscr{N}(0,1)$ if and only if
\begin{align} \label{Xn-clt}
    \mathbb{E}\left(X_n^*\right)^4\to 3 \qquad(n\to\infty).
\end{align}

\item [--] (Asymptotic Bernoulli distribution)
The sequence $\{X_n^*\}$ converges to Bernoulli random variable
assuming the two values $-1$ and $1$ with equal probabilities 
if and only if
\begin{align} \label{Xn-bernoulli}
    \mathbb{E}\left(X_n^*\right)^4\to 1\qquad(n\to\infty).
\end{align}
\end{description}
\end{thm}
This theorem implies that Gaussian and Bernoulli distributions are
in a certain sense extremal limit laws for the distribution of
$X_n$, maximizing and minimizing asymptotically the value of
the fourth moment $\mathbb{E}\left(X_n^*\right)^4$, respectively,
with other limit laws lying in between.

A standard example where Gaussian limit law arises is the number of
inversions in random permutations (or Kendall's $\tau$-statistic)
\[
    P_{\binom{n}2}(z)
    = \prod_{1\le k\le n}\frac{1+z+\cdots + z^{k-1}}{k}.
\]
A straightforward calculation shows that the fourth normalized
central moment has the form
\[
    3- \frac{9(6n^2+15n+16)}{25n(n-1)(n+1)},
\]
which implies the asymptotic normality by
Theorem~\ref{limit-law}; see \cite{feller45a}, Section
\ref{sec-app-i} for more details and examples.

On the other hand, a Bernoulli limit law results from the simple
example
\[
    P_{2n}(z)=\frac{1+z^{2n}}{2}.
\]

It is then natural to ask to which limit laws other than normal and
Bernoulli can the sequence of random variables $X_n^*$ converge. The
simplest such example is the uniform distribution
\[
    P_{2n}(z)=\frac{1+z+z^2+\cdots +z^{2n}}{2n+1};
\]
or, more generally, the finite sum of uniform distributions
\[
    P_{n_1+\cdots+n_k}(z)
    =\prod_{1\le j\le k}
    \frac{1+z+\cdots+z^{2n_j}}{2n_j+1}.
\]

Observe that the moment generating functions of the above three
distributions have the following representations.
\begin{description}
\item[--] Normal distribution: $e^{s^2/2}$;
\item[--] Bernoulli distribution (assuming $\pm1$ with equal
probability):
\[
    \frac{e^s+e^{-s}}{2} = \cos(is)
    =\prod_{k\ge1}\left(1+\frac{4s^2}{(2k-1)^2\pi^2}\right);
\]
\item[--] Uniform distribution (with zero mean and unit variance):
\[
    \frac{1}{2\sqrt{3}}
    \int_{\sqrt{3}}^{\sqrt{3}}e^{xs}\dd x
    =\frac{\sin(\sqrt{3}\,is)}{\sqrt{3}\,is}
    =\prod_{k\ge1}\left(1+\frac{3s^2}{k^2\pi^2}\right).
\]
\end{description}
Here we used the well-known expansions (see \cite{titchmarsh75a})
\begin{align*}
    \cos s = \prod_{k\ge1}
    \left(1-\frac{4s^2}{(2k-1)^2\pi^2}\right),\qquad
    \frac{\sin s}{s} =\prod_{k\ge1}
    \left(1-\frac{s^2}{k^2\pi^2}\right).
\end{align*}
We show that these are indeed special cases of a more general
representation theorem for the limit laws.

\begin{thm} \label{hadamard}
Let $\{X_n\}$ be a sequence of random variables whose probability
generating functions are polynomials with only roots of modulus one.
If the sequence $\{X_n^*\}$ converges to some limit distribution
$X$, then the moment generating function of $X$ is finite and has
the infinite-product representation
\begin{align}\label{Exs-ip}
    \mathbb{E}(e^{Xs})=e^{qs^2/2}
    \prod_{k\ge1}\left(1+\frac{q_k}2\,s^2\right),
\end{align}
where $q$ and the sequence $\{q_k\}$ are all non-negative numbers
such that
\[
    q+\sum_{k\ge1} q_k=1.
\]
\end{thm}
The above examples show that $q_k=\frac{8}{\pi^2(2k-1)^2}$ for
Bernoulli distribution and $q_k=\frac{6}{\pi^2k^2}$ for the uniform
distribution. More examples will be discussed below.

It remains open to characterize infinite-product representations of
the form \eqref{Exs-ip} that are themselves the moment generating
functions of limit laws of root-unitary polynomials. On the other
hand, many sufficient criteria for root-unitarity have been proposed
in the literature; see, for example, the book \cite{milovanovic94a}
and the recent papers \cite{lalin12a,suzuki12a} for more information and
references.

This paper is organized as follows. We first prove
Theorem~\ref{limit-law} in the next section when $n$ is even, and
then modify the proof to cover polynomials of odd degrees.
Theorem~\ref{hadamard} is then proved in
Section~\ref{sec-gll}. We then apply the results to many concrete
examples from the literature: Section~\ref{sec-app-i} for normal limit
laws and Section~\ref{sec-app-ii} for non-normal laws. A class of
polynomials with non-normal limit law is included in Appendix 
since the root-unitarity property has not yet been proved.  

\section{Moments and the two extremal limit distributions}

For convenience, we begin by considering (general) polynomials of
even degree with all their roots lying on the unit circle
\[
    P_{2n}(z)=\sum_{0\le k\le 2n}p_kz^k,
\]
where $p_k\ge 0$. To avoid triviality, we assume that not all
$p_k$'s are zero. Observe that if $|\rho|=1$ and $P(\rho)=0$, then
$P(\overline{\rho})=0$. If $\rho=1$, then its multiplicity must be
even since all other roots can be grouped in pairs and are symmetric
with respect to the real line. Thus our polynomials can be factored
as
\[
    P_{2n}(z)
    =\prod_{1\le j\le n}(z-\rho_j)(z-\overline{\rho_j}),
\]
where $|\rho_j|=1$ for $j=1,\dots n$. This factorization implies
that root-unitary polynomials are always self-inversive.
\begin{lem}
The coefficients of a root-unitary polynomial of even degree $2n$
are symmetric with respect to $n$, that is
\[
    p_{n-k}=p_{n+k}\qquad(0\le k\le n).
\]
\end{lem}
\begin{proof}
By replacing $z$ by $1/z$, we get
\[
\begin{split}
    \sum_{0\le k\le 2n} p_{2n-k}z^k &=z^{2n}P_{2n}(1/z)
    =\prod_{1\le j\le n}(1-z\rho_j)(1-z\overline{\rho_j})\\
    &=\prod_{1\le j\le n}(z-\rho_j)(z-\overline{\rho_j})
    =P_{2n}(z)=\sum_{0\le k\le 2n}p_kz^k.
\end{split}
\]
Taking the coefficients of $z^k$ on both sides, we obtain
$p_{2n-k}=p_k$ for $0\le k\le 2n$, which proves the lemma.
\end{proof}

\subsection{Random variables, moments and cumulants}

Since the coefficients of $P_{2n}(z)$ are nonnegative, we can define
a random variable $X_{2n}$ by
\[
    \mathbb{E}(z^{X_{2n}})=\frac{P_{2n}(z)}{P_{2n}(1)}.
\]

For convenience, we write $\rho_j=e^{i\phi_j}$ since $|\rho_j|=1$.
Then
\[
    (z-\rho_j)(z-\overline{\rho_j})
    =1-(\rho_j+\overline{\rho_j})z+z^2
    =1-2z\cos\phi_j +z^2.
\]
It follows that
\begin{align}\label{EzX2n}
    \mathbb{E}(z^{X_{2n}})
    =\prod_{1\le j\le n}\frac{1-2z\cos\phi_j+z^2}{2(1-\cos\phi_j)}.
\end{align}
Note that $\phi_j\not=0$ for $1\le j\le n$ since $P_{2n}(1)>0$.

It turns out that the mean values of such random variables are
identically $n$.
\begin{lem} For $n\ge1$
\begin{align}\label{EX2n}
    \mathbb{E}(X_{2n}) =n.
\end{align}
\end{lem}
\begin{proof}
By \eqref{EzX2n}, take derivative with respect to $z$ and then
substitute $z=1$.
\end{proof}
The relation \eqref{EX2n} indeed holds more generally for 
self-inversive polynomials; see, for example, \cite{sheil-small02a}.
\begin{cor} All odd central moments of $X_{2n}$ are zero
\[
    \mathbb{E}\bigl(X_{2n}-n\bigr)^{2m+1} = 0\qquad
    (m=0,1,\dots).
\]
\end{cor}
\begin{proof}
This follows from the symmetry of the coefficients $p_k$.
\end{proof}

For even moments, we look at the cumulants, which are defined as
\[
    \mathbb{E}(e^{(X_{2n}-n)s})
    =\exp\left(\sum_{m\ge1}
    \frac{\kappa_{m}(n)}{m!}\,s^{m}\right),
\]
where $\kappa_{2m+1}(n)=0$.
\begin{lem} The $2m$-th cumulant $\kappa_{2m}(n)$ of $X_{2n}$ is
given by
\begin{align} \label{kappa-2m}
    \kappa_{2m}(n) = (2m)!\sum_{1\le k\le m}
    \frac{(-1)^{k-1}}{k2^k}\,h_{m,k}S_{n,k}\qquad(m\ge1),
\end{align}
where $2^{2k}\sinh^{2k}(s/2)=\sum_{m\ge k} h_{m,k}s^{2m}$, with
$h_{k,k}=1$, and
\[
    S_{n,k}:=\sum_{1\le j\le n}\frac{1}{(1-\cos \phi_j)^k}.
\]
\end{lem}
\begin{proof}
By \eqref{EzX2n}, we have
\begin{align*}
    \log\frac{1-2e^s\cos\phi+e^{2s}}{2(1-\cos\phi)}
    &=\log\left(e^s+\frac{(e^s-1)^2}{2(1-\cos\phi)}\right)\\
    &=s+\log\left(1+2\frac{\sinh^2(s/2)}{1-\cos\phi}\right).
\end{align*}
Thus
\[
    \log\mathbb{E}(e^{(X_{2n}-n)s})
    =\sum_{1\le j\le n} \log\left(1+2\frac{\sinh^2(s/2)}
    {(1-\cos\phi_j)}\right),
\]
which implies \eqref{kappa-2m} by a direct expansion.
\end{proof}

\subsection{Variance and fourth central moment}

In particular, we obtain, from \eqref{kappa-2m},
\begin{align}
    \sigma_n^2:=\mathbb{V}(X_{2n})
    =\kappa_2(n)
    =\sum_{1\le j\le n}\frac{1}{1-\cos\phi_j}.
    \label{kappa-2}
\end{align}
\begin{lem} The variance satisfies the inequalities
\begin{align} \label{ineq-var}
    \frac n2 \le \sigma_n^2 \le n^2.
\end{align}
\end{lem}
\begin{proof}
The lower bound follows from \eqref{kappa-2} and the inequality
$1-\cos\phi_j\le 2$. The upper bound is also straightforward
\[
    \mathbb{V}(X_{2n})= \frac1{P_{2n}(1)}
    \sum_{0\le k\le 2n}p_k(k-n)^2\le n^2,
\]
which shows that the distance of the unit zeroes of $P_{2n}(z)$ to
the point $1$ is always larger than $c/n$, where $c>0$ is an
absolute constant.
\end{proof}

On the other hand, by the elementary inequalities
\begin{align} \label{ineq-cos}
    \frac2{\pi^2}\,t^2\le 1-\cos t \le
    \frac{t^2}2 \qquad(t\in[-\pi,\pi]),
\end{align}
we have
\[
    2 \le \frac{\sigma_n^2}
    {\sum_{1\le j\le n}\phi_j^{-2}}
    \le \frac{\pi^2}2.
\]

We now turn to the fourth central moment. Define
\begin{align} \label{omega-n}
    \omega_n :=\frac{1}{\sigma_n^4}\sum_{1\le j\le n}
    \frac{1}{(1-\cos\phi_j)^2}.
\end{align}
\begin{lem} $(i)$ For $n\ge1$,
\begin{align}\label{4th-mm-ub}
    1\le \mathbb{E}\left(\frac{X_{2n}-n}{\sigma_n}\right)^4
    \le 3-\frac{1}{2\sigma_n^2}<3.
\end{align}
$(\text{ii})$
\begin{align} \label{4th-mm-iff}
    \mathbb{E}\left(\frac{X_{2n}-n}{\sigma_n}\right)^4
    \to 3 \quad\text{iff}\quad \omega_n \to0.
\end{align}
\end{lem}
\begin{proof}
By definition and by \eqref{kappa-2m},
\begin{align}\label{fourth_moment}
    \mathbb{E}\left(\frac{X_{2n}-n}{\sigma}\right)^4
    =3+\frac{\kappa_4(n)}{\sigma_n^4}
    =3+\sigma_n^{-2}-3\omega_n.
\end{align}
Now
\[
    \sigma_n^{-2}-3\omega_n
    = -\frac1{\sigma_n^4}\sum_{1\le j\le n}
    \frac{2+\cos\phi_j}{(1-\cos\phi_j)^2}
    \le -\frac1{2\sigma_n^4}\sum_{1\le j\le n}\frac1{1-\cos\phi_j}
    = -\tfrac12\sigma_n^{-2}<0,
\]
proving the upper bound of \eqref{4th-mm-ub}. On the other hand,
since $1/\sigma_n\le \sqrt{2/n}$ (by \eqref{ineq-var}), we see that
\eqref{4th-mm-iff} also holds. It remains to prove the lower bound
of \eqref{4th-mm-ub}, which results directly from the Cauchy-Schwarz
inequality
\[
    1=\mathbb{E}\left(\frac{X_{2n}-n}
    {\sigma_n}\right)^2\le
    \left(\mathbb{E}\left(\frac{X_{2n}-n}
    {\sigma_n}\right)^4\right)^{1/2}.
\]
\end{proof}
By \eqref{ineq-cos}, we can replace the condition $\omega_n\to0$ by
\[
    \sum_{1\le j\le n} \phi_j^{-4}
    =o\Biggl(\Bigl(\sum_{1\le j\le n}
    \phi_j^{-2}\Bigr)^2\Biggr) .
\]
On the other hand, $\mathbb{E}((X_n-n)/\sigma_n)^4\to3$ is
equivalent to $\kappa_4(n)/\kappa_2^2(n)\to0$; the latter condition
is in many cases easier to manipulate.

Note that \eqref{4th-mm-ub} proves \eqref{4th-mm-bds} when $n$ is
even.

\subsection{Estimates for the moment generating functions}

\begin{lem}\label{est-cf}
Assume $\omega_n\to0$.
For all $s\in \mathbb{C}$ such that $|s|\le
\min\{\sigma_n,\omega_n ^{-1/4}\}/4$, we have
\begin{align} \label{mgf-est1}
    \mathbb{E}(e^{(X_{2n}-n)s/\sigma_n})
    =\exp\left(\frac{s^2}{2}+O\left(\frac{|s|^3}
    {\sigma_n}+\omega_n |s|^4\right)\right).
\end{align}
\end{lem}
\begin{proof}
By \eqref{EzX2n},
\[
    \log\mathbb{E}(e^{X_{2n}s/\sigma_n})
    =\sum_{1\le j\le n}\log\left( 1+(e^{s/\sigma_n}-1)+
    \frac{(e^{s/\sigma_n}-1)^2}{2(1-\cos \phi_j)}\right)
\]
Note that, by \eqref{kappa-2},
\[
    \sigma_n^2\ge \max_{1\le j \le n}\frac{1}{1-\cos \phi_j}.
\]
From the definition \eqref{omega-n} of $\omega_n$, we also have
\[
    \frac{1}{\sigma_n^4}\max_{1\le j \le n}
    \frac{1}{(1-\cos \phi_j)^2}\le \omega_n,
\]
which means that
\begin{align} \label{one-cos}
    \max_{1\le j \le n}\frac{1}{1-\cos \phi_j}
    \le \sigma_n^2\sqrt{\omega_n }.
\end{align}
Thus
\[
    \left|e^{s/\sigma_n}-1 + \frac{
    (e^{s/\sigma_n}-1)^2}{2(1-\cos \phi_j)}
    \right|\le e^{2|s|/\sigma_n}
    \left(\frac{|s|}{\sigma_n}+|s|^2\sqrt{\omega_n }\right)
\]
Since $\omega_n \to 0$ by assumption, the right-hand side is less
than, say $2/3$, for large enough $n$ when $|s|$ remains bounded.
Thus we can use the Taylor expansion of $\log(1+w)$ and obtain
\[
\begin{split}
    &\log\left( 1+(e^{s/\sigma_n}-1) +
    \frac{(e^{s/\sigma_n}-1)^2}{2(1-\cos \phi_j)}\right) \\
    &\qquad =\frac{s}{\sigma_n}
    +\frac{s^2}{2\sigma_n^2(1-\cos \phi_j)}
    +O\left(\frac{|s|^3}{\sigma_n^3(1-\cos\phi_j)}
    +\frac{|s|^4}{\sigma_n^4(1-\cos\phi_j)^2}
    +\frac{|s|^6}{\sigma_n^6(1-\cos\phi_j)^3}\right).
\end{split}
\]
By \eqref{one-cos}
\[
    \frac{|s|^6}{\sigma_n^6(1-\cos\phi_j)^3}
    \le  \frac{|s|^6\sqrt{\omega_n }}{\sigma_n^4(1-\cos\phi_j)^2}.
\]
It follows, after summing over all $j$, that
\[
    \log \mathbb{E}(e^{X_{2n}s/\sigma_n})
    =\frac{s^2}{2}+\frac{ns}{\sigma_n}
    +O\left(\frac{|s|^3}{\sigma_n}+(|s|^4+|s|^6
    \sqrt{\omega_n })\omega_n\right).
\]
Now if $|s|\le\min\{\sigma_n,\omega_n ^{-1/4}\}/4$, then
$\omega_n^{3/2}|s|^6 \leqslant \omega_n|s|^4/16$, and this proves
\eqref{mgf-est1}.
\end{proof}

\begin{lem}\label{est-mgf}
For $s\in \mathbb{R}$, the inequality
\begin{align} \label{mgf-ineq}
    \mathbb{E}(e^{(X_{2n}-n)s/\sigma_n})
    \le \exp\left(\tfrac32\, s^2e^{2s/\sigma_n}\right)
\end{align}
holds.
\end{lem}
\begin{proof}
By \eqref{EzX2n} and the elementary inequality $1+y\le e^{y}$ for
real $y$, we obtain
\[
\begin{split}
    \mathbb{E}(z^{X_{2n}})&=\prod_{1\le j\le n}
    \left(1+(z-1)+\frac{(z-1)^2}{2(1-\cos \phi_j)}\right)\\
    &\le \prod_{1\le j\le n}\exp\left(z-1+\frac{(z-1)^2}
    {2(1-\cos \phi_j)}\right)\\
    &=e^{n(z-1)+\sigma_n^2(z-1)^2/2}.
\end{split}
\]
Thus
\[
\begin{split}
    \mathbb{E}(e^{(X_{2n}-n)s/\sigma_n})
    &\le \exp\left(n\left(e^{s/\sigma_n}-1
    -\frac s{\sigma_n}\right)+
    \tfrac12\sigma_n^2(e^{s/\sigma_n}-1)^2\right)\\
    &\le\exp\left(\frac{n}{2\sigma_n^2}
    \,s^2e^{s/\sigma_n}+\frac{s^2}2\, e^{2s/\sigma_n}\right),
\end{split}
\]
and \eqref{mgf-ineq} follows from the inequality $n/\sigma_n^2\le
2$.
\end{proof}

\subsection{Normal limit law}

We now prove the second part of Theorem~\ref{limit-law} in the
case of polynomials of even degree, namely, $\{(X_n-n)/\sigma_n\}$
converges in distribution and with all moments to the standard
normal distribution if and only if 
\[
    \mathbb{E}\left(\frac{X_{2n}-n}{\sigma_n}\right)^4\to3.
\]

\begin{proof} Consider first the sufficiency part. 
By \eqref{4th-mm-iff}, $\omega_n \to 0$, and we can apply the
estimate \eqref{mgf-est1}, implying the convergence in distribution
of $(X_{2n}-n)/\sigma_n$ to $\mathscr{N}(0,1)$.

On the other hand, by Lemma \ref{est-mgf},
\[
    \mathbb{E}(e^{(X_{2n}-n)s/\sigma_n})
    =\sum_{m\ge0}\left(\frac{X_{2n}-n}{\sigma_n}\right)^{2m}
    \frac{s^{2m}}{(2m)!}
    \le e^{\frac32s^2e^{2s/\sigma_n}}.
\]
Taking $s=1$, we conclude that all normalized central moments of 
$X_{2n}$ are bounded above by
\[
    \mathbb{E}\left(\frac{X_{2n}-n}{\sigma_n}\right)^{2m}
    \le (2m)!e^{\frac32e^{2/\sigma_n}}.
\]
Thus we also have convergence of all moments.

For the necessity, we see that if $\{(X_{2n}-n)/\sigma_n\}$ converges 
in distribution to $\mathscr{N}(0,1)$, then the the fact that the
moments of $(X_{2n}-n)/\sigma_n$ are all bounded implies that all
the normalized central moments of $X_{2n}$ converge to the moments of 
the standard normal distribution; in particular, the fourth normalized
central moments converge to $3$.
\end{proof}

\subsection{Bernoulli limit law}

We now examine the case when the fourth moment converges to
the smallest possible value,  that is
\begin{align} \label{4th-mm-1}
    \mathbb{E}\left(\frac{X_{2n}-n}
    {\sigma_n}\right)^4\to 1.
\end{align}
Note that
\[
    \mathbb{V}\left(\frac{X_{2n}-n}
    {\sigma_n}\right)^2=
    \mathbb{E}\left(\left(\frac{X_{2n}-n}
    {\sigma_n}\right)^2-1\right)^2=
    \mathbb{E}\left(\frac{X_{2n}-n}
    {\sigma_n}\right)^4-1.
\]
If \eqref{4th-mm-1} holds, then by Chebyshev's inequality, we see
that
\[
    \mathbb{P}\left(\frac{X_{2n}-n}
    {\sigma_n}\in
    (-1-\varepsilon,-1+\varepsilon)\bigcup
    (1-\varepsilon,1+\varepsilon)\right)\to 1,
\]
for any $\varepsilon>0$. By symmetry of the random variable
$X_{2n}-n$
\[
    \mathbb{P}\left(\frac{X_{2n}-n}
    {\sigma_n}\in (-1-\varepsilon,-1+\varepsilon)\right)
    =\mathbb{P}\left(\frac{X_{2n}-n}
    {\sigma_n}\in (1-\varepsilon,1+\varepsilon)\right).
\]
We conclude that the distributions of $(X_{2n}-n)/\sigma_n$ converge
to a Bernoulli distribution that assumes the two values $1$ and $-1$
with equal probability.

\subsection{Polynomials of odd degree}

To complete the proof of Theorem~\ref{limit-law}, we need to 
address the situation of odd-degree polynomials.

Assume $Q_{2n-1}(z)$ is a root-unitary polynomial of degree $2n-1$
with non-negative coefficients. If we multiply it by the factor
$1+z$, then the resulting polynomial
\[
    P_{2n}(z)=(1+z)Q_{2n-1}(z)
\]
remains root-unitary with non-negative coefficients. This means that
the moment generating functions of the corresponding random
variables $\mathbb{E}(e^{Y_{2n-1}s}):=Q_{2n-1}(e^s)/Q_{2n-1}(1)$ and
$\mathbb{E}(e^{X_{2n}s}):=P_{2n}(e^s)/P_{2n}(1)$ are connected by the
identity
\[
    \mathbb{E}(e^{X_{2n}s})
    =\frac{1+e^s}{2}\,\mathbb{E}(e^{Y_{2n-1}s}).
\]
This leads to the relation
\begin{align} \label{X-Y}
    X_{2n}\stackrel{d}{=}Y_{2n-1}+B,
\end{align}
where $B$ is independent of $Y_{2n-1}$ and takes the values $0$ and
$1$ with equal probability. Thus
\begin{align}
    \mathbb{E}(Y_{2n-1}) &= \mathbb{E}(X_{2n})-\tfrac{1}{2}
    =n-\tfrac12,\nonumber \\
    \mathbb{V}(Y_{2n-1}) &= \mathbb{V}(X_{2n})-\tfrac{1}{4}
    = \sigma_n^2-\tfrac14, \label{V-X-Y}
\end{align}
and
\begin{equation}
\label{Y-4th-mm}
    \mathbb{E}\left(Y_{2n-1}-\mathbb{E}(Y_{2n-1})\right)^4
    = \mathbb{E}(X_{2n}-n)^4 -\tfrac32 \sigma_n^2
    +\tfrac 5{16}.
\end{equation}
Thus we obtain
\begin{align*}
    \mathbb{E}\left(\frac{Y_{2n-1}-\mathbb{E}(Y_{2n-1})}
    {\sqrt{\mathbb{V}(Y_{2n-1})}}\right)^4
    &\le \frac{\sigma_n^4}{(\sigma_n^2-\frac14)^2}\,
    \mathbb{E}\left(\frac{X_{2n}-n}
    {\sqrt{\mathbb{V}(X_{2n})}}\right)^4 -
    \frac{\frac32\sigma_n^2-\frac5{16}}
    {(\sigma_n^2-\frac14)^2},
\end{align*}
which, by \eqref{4th-mm-ub}, is bounded above by
\[
    \frac{\sigma_n^4}{(\sigma_n^2-\frac14)^2}
     \left(3-\frac1{\sigma_n^2}\right)
    -\frac{\frac32\sigma_n^2-\frac5{16}}
    {(\sigma_n^2-\frac14)^2}=3-\frac1{\sigma_n^2}
    -\frac{6\sigma_n^2-1}
    {\sigma_n^2(4\sigma_n^2-1)^2}\le 3-\sigma_n^{-2}<3.
\]
Thus the fourth normalized central moment is bounded above by
$3$; the lower bound follows from the same Cauchy-Schwarz inequality
used in the even-degree cases.

On the other hand, since (again by \eqref{X-Y})
\[
    \frac{ X_{2n}-\mathbb{E}(X_{2n})}{\sqrt{\mathbb{V}(X_{2n})}}
    \stackrel{d}{=}\frac{\sqrt{\mathbb{V}(Y_{2n-1})}}
    {\sqrt{\mathbb{V}(X_{2n})}}
    \cdot \frac{Y_{2n-1}-\mathbb{E}(Y_{2n-1})}
    {\sqrt{\mathbb{V}(Y_{2n-1})}}+\frac{B-\frac12}
    {\sqrt{\mathbb{V}(X_{2n})}},
\]
we have, by (\ref{V-X-Y}),
\begin{equation}\label{X-Y-1}
    \frac{ X_{2n}-\mathbb{E}(X_{2n})}{\sqrt{\mathbb{V}(X_{2n})}}
    \stackrel{d}{=}
    \frac{Y_{2n-1}-\mathbb{E}(Y_{2n-1})}
    {\sqrt{\mathbb{V}(Y_{2n-1})}}
    \left(1+O\left(\frac{1}{\sqrt{n}}\right)\right)
    +O\left(\sigma_n^{-1}\right).
\end{equation}
The last identity implies that both sides converge to the same
limit law.

Assume that the fourth central moment of $Y_{2n-1}$ satisfies
\begin{align}\label{Y-2n-1}
    \mathbb{E}\left(\frac{Y_{2n-1}-\mathbb{E}(Y_{2n-1})}
    {\sqrt{\mathbb{V}(Y_{2n-1})}}\right)^4\to 3.
\end{align}
Then, by (\ref{Y-4th-mm}), we obtain
\[
    \mathbb{E}\left(\frac{X_{2n}-\mathbb{E}(X_{2n})}
    {\sqrt{\mathbb{V}(X_{2n})}}\right)^4
    =\left(\frac{\mathbb{V}(Y_{2n-1})}{\mathbb{V}(X_{2n})}\right)^2
    \mathbb{E}\left(\frac{Y_{2n-1}-\mathbb{E}(Y_{2n-1})}
    {\sqrt{\mathbb{V}(Y_{2n-1})}}\right)^4
    +O\left(\sigma_n^{-1}\right).
\]
Thus the left-hand side also tends to $3$ and, consequently,
$X_{2n}$ is asymptotically normally distributed. The asymptotic
distribution of $X_{2n}$ then implies, by
\eqref{X-Y-1}, that of $Y_{2n-1}$.

The proof for the Bernoulli case is similar and is omitted.

\section{The infinite-product representation for general limit laws}
\label{sec-gll}

We first prove Theorem \ref{hadamard} in this section, and
then mention some of its consequences.

\subsection{Proof of Theorem \ref{hadamard}}

It suffices to consider only the sequence of polynomials of even
degree. The symmetry of distribution of the limit law $X$ follows
from the symmetry of coefficients of polynomials $P_{2n}(z)$. The
inequality \eqref{mgf-ineq} for the moment generating function of
$(X_{2n}-n)/\sigma_n$ implies that the moment generating function of
the limit distribution $X$ is also finite, and thus $X$ is uniquely
determined by its moments. This means that the sequence
$\{(X_{2n}-n)/\sigma_n\}$ converges in distribution to $X$ as $n\to
\infty$ if and only if
\begin{equation*}
    \mathbb{E}\left(\frac{X_{2n}-n}{\sigma_n}\right)^m
    \to \mathbb{E}(X^m)\qquad(m\ge0),
\end{equation*}
as $n\to \infty$. Thus the cumulant $\bar{\kappa}_m(n)$ of
$(X_{2n}-n)/\sigma_n$ of order $m$ also converges to the cumulant of
$X$ of order $m$ for $m\ge1$. Note that $\bar{\kappa}_{2m+1}(n)=0$
for $m\ge0$ and (see \eqref{kappa-2m})
\[
    \bar{\kappa}_{2m}(n) = \sigma_n^{-2m}\kappa_{2m}(n)
    = \frac{(2m)!}{\sigma_n^{2m}}
    \sum_{1\le k\le m}\frac{(-1)^{k-1}}{k2^k}\,h_{m,k}S_{n,k}.
\]
Since $S_{n,k}\le \sigma_n^{2k}$, we the deduce that
\[
    \frac{\bar{\kappa}_{2m}(n)}{(2m)!}=
    \frac{(-1)^{m-1}}
    {m2^m}\cdot \frac{S_{n,m}}{\sigma_n^{2m}}
    +O(\sigma_n^{-2}),
\]
for any fixed $m$. Now $\sigma_n\to \infty$, we conclude that
\begin{equation}\label{lim-cumu}
    \frac{\kappa_{2m}}{(2m)!}
    =\lim_{n\to \infty}\frac{\bar{\kappa}_{2m}(n)}{(2m)!}
    =\frac{(-1)^{m-1}}{m2^m}\lim_{n\to \infty}
    \frac{S_{n,m}}{\sigma_n^{2m}}.
\end{equation}
We now introduce the distribution function
\[
    F_n(x):=
    \sum_{\frac{1}{\sigma_n^2(1-\cos\phi_j)}<x}
    \frac{1}{\sigma_n^2(1-\cos\phi_j)},
\]
with support in the unit interval. Then
\[
    \frac{S_{n,N}}{\sigma_n^{2N}}=\int_0^1x^{N-1}\dd F_n(x).
\]
The fact that the left-hand side of the above expression has a limit
(\ref{lim-cumu}) implies that the corresponding sequence
of distribution functions $F_n(x)$ also converges weakly to some
limit distribution function $F(x)$. Therefore
\[
    \lim_{n\to \infty}\frac{S_{n,N}}{\sigma_n^{2N}}
    =\int_0^1x^{N-1}\dd F(x),
\]
which implies that the cumulants of the limit distribution $X$ can
be expressed as
\[
    \frac{\bar{\kappa}_{2m}}{(2m)!}=\lim_{n\to \infty}
    \frac{\bar{\kappa}_{2m}(n)}{(2m)!}
    =\frac{(-1)^{m-1}}{m2^m}\lim_{n\to \infty}
    \frac{S_{n,m}}{\sigma_n^{2m}}
    =\frac{(-1)^{m-1}}{m2^m}\int_0^1x^{m-1}\dd F(x).
\]
It follows that
\begin{align}\label{MGF}
    \begin{split}
    \mathbb{E}(e^{Xs})&=\exp\left(\sum_{m\ge1}
    \frac{\kappa_{2m}}{(2m)!}s^{2m}\right)\\
    &=\exp\left(\sum_{m\ge1} \frac{(-1)^{m-1}s^{2m}}
    {m2^m}\int_0^1x^{m-1}\dd F(x)\right)\\
    &=\exp\left(\int_{0}^1\frac{\log
    \left(1+xs^2/2\right)}{x}\,\dd F(x)\right).
    \end{split}
\end{align}

Note that the distribution function $F_n(x)$ has no more than
$\lfloor1/\varepsilon\rfloor$ points of discontinuity in the
interval $[\varepsilon,1]$ if $\varepsilon>0$. Thus the weak limit
$F(x)$ of the sequence of $F_n(x)$ has the same property: $F(x)$ has
no more than $\lfloor1/\varepsilon\rfloor$ points of discontinuity
$q_k$ in the interval $[\varepsilon,1]$, where $q_k$ is the limit of
certain points of discontinuity of function $F_n(x)$. This means
that $F(x)$ is a distribution function of the form
\[
    F(x)=\begin{cases}
    q+\sum_{q_k<x}q_k, &\text{if $x\ge0$,}\\
    0, & \text{if $x<0$},
    \end{cases}
\]
where $q_k>0$ with $\sum_{k\ge1} q_k=1-q$. Here $q$ equals the jump
of the function $F(x)$ at zero. Thus
\[
    \int_{0}^1\frac{\log\left(1+xs^2/2\right)}{x}\dd F(x)
    =\frac{q}{2}\,s^2+\sum_{k\ge1}
    \log\left(1+\frac{q_k}{2}\,s^2\right).
\]
Substituting this expression into (\ref{MGF}), we obtain
\eqref{Exs-ip}. This completes the proof of Theorem
\ref{hadamard}.

\subsection{An alternative proof of Theorem \ref{hadamard}}

A less elementary proof of Theorem \ref{hadamard} relies on
the Hadamard factorization theorem (see \cite[Ch.\
8]{titchmarsh75a}; see also \cite{newman74a} for a similar context).
Indeed, assume that $(X_{2n}-n)/\sigma_n$ converges in distribution
to some limit law $X$, then the inequality \eqref{mgf-ineq} implies
that
\[
    \left|\mathbb{E}(e^{Xs})\right|
    \le  e^{3|s|^2/2} \qquad(s\in\mathbb{C}).
\]
In other words, it is an entire function of order $2$. Hadamard's
factorization theorem then implies that such a function can be
represented as an infinite product
\[
    \mathbb{E}(e^{Xs})
    =e^{As^2+Bs}\prod_{\rho}
    \biggl(1-\frac{s}{\rho}\biggr)e^{s/\rho},
\]
where $\rho$ ranges over all zeros of the function of the left-hand
side. On the other hand, the fact that all zeroes of the functions
$\mathbb{E}(e^{(X_{2n}-n)s/\sigma_n})$ are symmetrically located on the
imaginary line implies the same property for $\mathbb{E}(e^{Xs})$.
This yields
\[
    \mathbb{E}(e^{Xs})=e^{As^2+Bs}
    \prod_{k\ge1}\biggl(1+\frac{s^2}{t_k^2}\biggr),
\]
for some real sequence $t_k>0$. Now $\mathbb{E}(X)=0$ implies that $B=0$. 
Also $\mathbb{E}(X^2)=1$ leads to
\[
    A+\sum_{k\ge1}t_k^{-2}=1.
\]
Denoting by $q=2A$ and $q_k=2/t_k^2$, we obtain the representation 
\eqref{Exs-ip}.

\subsection{Implications of the infinite-product factorization}

By \eqref{Exs-ip},
\[
    \kappa_{2m} = \frac{(-1)^{m-1}}{m2^m}\,\sum_{j\ge1} q_j^m
    \qquad(m\ge2).
\]
This yields the sign-alternating property for the sequence
$\{\kappa_{2m}\}$.
\begin{cor} If $X$ is not the normal law, then all even cumulants are
non-zero and have alternating signs
\[
    (-1)^{m-1}\kappa_{2m}>0 \qquad(m\ge1).
\]
\end{cor}

\begin{cor}
\[
    1\le \mathbb{E}(X^4) \le 3.
\]
\end{cor}
\begin{proof}
By \eqref{Exs-ip},
\begin{align} \label{EX4}
    \mathbb{E}(X^4)
    =3\biggl(1-\sum_{j\ge1} q_j^2\biggr),
\end{align}
which implies the upper bound; the lower bound follows directly from
Cauchy-Schwarz inequality $1=\mathbb{E}(X^2)\le \sqrt{\mathbb{E}(X^4)}$.
\end{proof}
\begin{cor}
The standard normal distribution is the only distribution for which
the fourth moment reaches the maximum value $3$ in the class of
distributions that are the limits of random variables whose
probability generating functions are root-unitary polynomials;
similarly, the Bernoulli distribution assuming $\pm1$ with
probability $1/2$ each is the only distribution whose fourth moment
reaches the minimum value $1$ in the same class of distributions.
\end{cor}
\begin{proof}
Note that the standard normal law corresponds to the choices $q=1$
and $q_j\equiv 0$, the first part of the corollary follows then from
\eqref{EX4}.

For the lower bound, assume that $Y$ is a symmetric distribution
such that $\mathbb{E}(Y)=0$ and $\mathbb{E}(Y^2)=\mathbb{E}(Y^4)=1$. 
Then
\[
    \mathbb{V}(Y^2)=\mathbb{E}(Y^2-1)^2
    =\mathbb{E}(Y^4-2Y^2+1)=0.
\]
This means that $Y$ can only assume two values
$\mathbb{P}(Y\in\{-1,1\})=1$. The symmetry of $Y$ now implies that
$Y$ assumes the values $1$ and $-1$ with equal probabilities.
\end{proof}

\begin{rem}
The uniqueness of the standard normal and Bernoulli laws also implies 
that a sequence of random variables $\{X_n\}$ converges to normal or 
Bernoulli if and only if its fourth normalized central moment 
converges to $3$ or to $1$, respectively. This provides an alternative 
proof of the last two statements of Theorem \ref{limit-law}.
\end{rem}

\section{Applications. I. Normal limit law}
\label{sec-app-i}

We consider in this section applications of our results in the
situations when the limit law is normal.

\subsection{A simple framework}
Our starting point is the polynomials of the form
\begin{align} \label{PNz}
    P_n(z)=\frac{(1-z^{b_1})(1-z^{b_2})
    \cdots (1-z^{b_N})}{(1-z^{a_1})(1-z^{a_2})
    \cdots (1-z^{a_N})},
\end{align}
where $a_j$, $b_j$ are non-negative integers that may depend
themselves on $N$ and 
\[
    n := \sum_{1\le j\le N}(b_j-a_j).
\]
We assume that $P_n(z)$ has only nonnegative
coefficients. Such a simple form arises in a large number of diverse
contexts, some of which will be examined below. In particular, it
was studied in the recent paper \cite{chen08a}.

We now consider a sequence of random variables $X_n$ defined by
\[
    \mathbb{E}(z^{X_n})=\frac{P_n(z)}{P_n(1)}.
\]
We have
\[
    \frac{P_n(e^s)}{P_n(1)}
    =\exp\Bigl(\sum_{m\ge1}\frac{\kappa_{N,m}}{m!}s^m\Bigr),
\]
where
\[
    \kappa_{N,m} = \frac{(-1)^m}{m}\,B_m
    \sum_{1\le j\le N}(b_j^m-a_j^m)\qquad(m\ge1),
\]
the $B_m$'s being the Bernoulli numbers. Note that $B_{2m+1}=0$ for
$m\ge1$.

An application of Theorem~\ref{limit-law} yields the following 
result. 
\begin{thm} \label{thm-clt} The sequence of the random variables
$(X_n-\mathbb{E}(X_n))/\sqrt{\mathbb{V}(X_n)}$ converges to the standard
normal distribution if and only if the following cumulant condition
holds
\begin{align}\label{cc}
    \lim_{N\to \infty}\frac{\kappa_{N,4}}{\kappa_{N,2}^2}
    =\frac{144}{120}\lim_{N\to \infty}
    \frac{\sum_{1\le j\le N}(b_j^4-a_j^4)}
    {\left(\sum_{1\le j\le N}(b_j^2-a_j^2)\right)^2}=0.
\end{align}
\end{thm}

The cumulant condition largely simplifies the sufficient condition
given by \cite{chen08a}, where they require the convergence of all
cumulants
\[
    \frac{\kappa_{N,2m}}{\kappa_{N,2}^m}\to 0 \qquad(m\ge2),
\]
following the proof used by \cite{sachkov97a}. See also 
\cite{janson88a} for a related framework. 

\subsection{Applications of Theorem~\ref{thm-clt} }

Theorem~\ref{thm-clt} can be applied to a large number of examples. 
Many other examples related to Poincar\'e polynomials, rank statistics, 
and integer partitions can be found in the literature; see, for 
example, \cite{akyildiz04a,andrews76a,van-de-wiel99a} and the 
references therein. 

\paragraph{Inversions in permutations}
The generating polynomial for the number of inversions in a
permutation of $n$ elements (or Kendall's $\tau$ statistic) 
is given by
\[
    \prod_{1\le j\le n}\frac{1-z^j}{1-z}.
\]
In this case, the cumulant condition \eqref{cc} has the form
\[
    \frac{\sum_{1\le j\le n}(j^4-1)}
    {\left(\sum_{1\le j\le n}(j^2-1)\right)^2}=O(n^{-1}).
\]
Thus the number of inversions in random permutations is
asymptotically normally distributed; see \cite{feller45a},
\cite{sachkov97a}; see also 
\cite{cronholm65a,louchard03a,margolius01a}.

\paragraph{Number of inversions in Stirling permutations} 
In this case, we have the polynomial (see \cite{park94a})
\[
    \prod_{1\le j\le n}
    \frac{1-z^{r+(j-1)r^2}}{1-z^r}\qquad(r\ge1),
\]
and the cumulant condition \eqref{cc} is of order
\[
    \frac{\kappa_{n,4}}{\kappa_{n,2}^2}
    =\frac{\sum_{0\le j<n}\bigl((r+jr^2)^4-1\bigr)}
    {\left(\sum_{0\le j<n}\bigl((r+jr^2)^2-1\bigr)\right)^2}
    =O(n^{-1}).
\]
Consequently, the number of inversions in random Stirling
permutations is asymptotically normally distributed.

\paragraph{Gaussian polynomials} The generating function for
the number $p(n,m,j)$ of partitions of integer $j$ into at most
$m$ parts, each $\le n$, is given by (see e.g.
\cite{andrews76a})
\[
    \sum_{0\le j\le nm}p(n,m,j)z^j
    = \prod_{1\le j\le n} \frac{1-z^{j+m}}{1-z^j}.
\]
Then the cumulant condition has the form
\[
    \frac{\sum_{1\le j\le n}((m+j)^4-j^4)}
    {\left(\sum_{1\le j\le n}((m+j)^2-j^2)\right)^2}
    =O\left(\frac{1}{m}+\frac{1}{n}\right).
\]
This means that the coefficients of Gaussian polynomials are
normally distributed if both $n,m\to \infty$; see
\cite{mann47a,takacs86a}. More examples can be found in 
\cite{andrews76a}. 

\paragraph{Mahonian statistics} 
In this case the polynomials are equal to the general $q$-multinomial 
coefficients (see \cite{canfield11a} and \cite{canfield12a})
\[
    P_{n}(z) 
    =\frac{\prod_{1\le j\le a_1+\cdots+a_m}(1-z^j)}
    {\prod_{1\le j\le m}\prod_{1\le i\le a_j}(1-z^i)},
\]
where $n = \sum_{2\le k\le m}a_k\sum_{1\le j<k} a_j$. 
By symmetry, we can assume that $a_1\ge \cdots \ge a_m$. 
Then the cumulant condition \eqref{cc} becomes
\[\begin{split}
    \frac{\sum_{1\le j\le a_1+\cdots+a_m}i^4
    -\sum_{1\le j\le m}\sum_{1\le i\le a_j}i^4}
    {\left(\sum_{1\le j\le a_1+\cdots+a_m}i^2
    -\sum_{1\le j\le m}\sum_{1\le i\le a_j}i^2\right)^2} 
    =\frac{f_4(a_1+\cdots+a_m)
    -\sum_{1\le j\le m}f_4(a_j)}{\bigl(f_2(a_1+\cdots+a_m)
    -\sum_{1\le j\le m}f_2(a_j)\bigr)^2},
\end{split}\]
where $f_2(x)=(2x^3+3x^2+x)/6$ and $f_4(x)=(6x^5+15x^4+10x^3-x)/30$.
By induction, $(a_1+\cdots+a_m)^k -a_1^k-\cdots -a_m^k$ is nonnegative
and is nondecreasing in $k\ge1$. Thus the right-hand side 
is bounded above by 
\[\begin{split}
    &\frac{9\cdot 31}{30}\frac{(a_1+\cdots+a_m)^5
    -a_1^5-\cdots-a_m^5}{\bigl((a_1+\cdots+a_m)^3
    -a_1^3-\cdots-a_m^3\bigr)^2} \\
    &\qquad =O\left(\frac{a_1+\cdots+a_m}
    {\sum_{1\le i<j\le m}a_ia_j}\right)
    =O\left(\frac{a_1+\cdots+a_m}
    {a_1(a_2+a_3+\cdots+a_m)}\right)\\
    &\qquad=O\left(\frac{1}{a_2+a_3+\cdots+a_m}+
    \frac{1}{a_1}\right),
\end{split}
\]
where we use the estimates
\[\begin{split}
    (a_1+\cdots+a_m)^3-a_1^3-\cdots-a_m^3
    &\asymp (a_1+\cdots+a_m)
    \sum_{1\le i<j\le m}a_ia_j,\\
    (a_1+\cdots+a_m)^5
    -a_1^5-\cdots-a_m^5
    &\asymp (a_1+\cdots+a_m)^3
    \sum_{1\le i<j\le m}a_ia_j.
\end{split}
\]
Thus we arrive at the same conditions as those in \cite{canfield11a}
\[
    a_1\to \infty \quad \hbox{and}
    \quad a_2+a_3+\cdots+a_m\to \infty,
\]
for the asymptotic normality of the coefficients of $P_n(z)$ when
$a_1\ge a_2\ge \cdots \ge a_m$.

\paragraph{Generalized $q$-Catalan numbers} The generating function
has the form
\[
    \prod_{2\le j\le n}\frac{1-z^{(m-1)n+j}}{1-z^j},
\]
and the cumulant condition \eqref{cc} also holds
\[\begin{split}
    \frac{\sum_{2\le j\le n}\bigl(((m-1)n+j)^4-j^4\bigr)}
    {\left(\sum_{2\le j\le n}
    \bigl(((m-1)n+j)^2-j^2\bigr)\right)^2}
    \le \frac{\sum_{2\le j\le n}(2mn)^4}{\left(\sum_{2\le j\le n}
    (m-1)^2n^2\right)^2}=O\left(n^{-1}\right),
\end{split}
\]
which means that the generalized $q$-Catalan numbers are 
asymptotically normally distributed, uniformly for all $m\ge 2$. 
This result was previously proved by \cite{chen08a}.

\paragraph{Sums of uniform discrete distributions} Let $X_n$ be the
sum of $N$ independent, integer-valued random variables
\[
    X_n :=J_1+J_2+\cdots+J_N,
\]
where $J_k$ is a uniform distribution on the set $\{0,1,2,\ldots,
d_k-1\}$ with $d_k\ge2$, and $n=\sum_{1\le j\le N}(d_j-1)$. 
Then the corresponding probability
generating function $\mathbb{E}(z^{X_n})$ is equal, up to a
normalizing constant, to
\[
    P_n(z)=\prod_{1\le j\le N} \frac{1-z^{d_j}}{1-z},
\]
which means that $X_n$ is asymptotically normal if and only
if
\[
    \frac{\sum_{1\le j\le N}(d_j^4-1)}
    {\left(\sum_{1\le j\le N}(d_j^2-1)\right)^2}\to 0.
\]
Since by our assumption $d_j\ge 2$, we have $d_j-1\asymp d_j$ and
thus we can simplify our necessary and sufficient condition for
asymptotic normality as
\begin{equation}\label{uniform_norm}
    \frac{d_1^4+d_2^4+\cdots+d_N^4}
    {(d_1^2+d_2^2+\cdots+d_N^2)^2}\to 0 \qquad(N\to\infty).
\end{equation}
Note that $d_j$ here can depend on $N$. The continuous version of
this problem with $J_k$ being uniformly distributed on the intervals
$[0,d_j]$ was considered in \cite{olds52a}. The corresponding
necessary and sufficient condition obtained in this paper was
\[
    \frac{\max_{1\le j\le N}d_j}
    {\sqrt{d_1^2+d_2^2+\cdots+d_N^2}}\to 0
\]
which is equivalent to condition (\ref{uniform_norm}).

\paragraph{Number of inversions in bimodal permutations}
A permutation $\sigma=(s_1,s_2,\ldots,s_n)$ of $n$ numbers
$1,2,3,\ldots,n$ is said to be of a shape $(i,k-j,j,l)$ if the first
$i$ numbers in the permutation are decreasing $s_1>s_2>\cdots>s_i$,
the next $k-j$ numbers are increasing $s_{i+1}> s_2> \cdots>
s_{i+k-j}$, then followed by $j$ increasing and $l$ decreasing
numbers. Assume that $\sigma$ is chosen with equal probability among
all permutations of shape $(i,k-j,j,l)$. Then its number of
inversions becomes a random variable $I_n = I_n(i,k-j,j,l)$. The 
probability generating function of $I_n$ is, up to some 
constant, of the form (see \cite{bohm05a})
\[
    P_n(i,k,l,j;z)=z^{\binom{i}{2}+\binom{j}{2}}
    \Biggl(\prod_{1\le \nu\le i}\frac{1-z^{k+\nu}}{1-z^\nu}\Biggr)
    \Biggl(\prod_{1\le \nu\le l}\frac{1-z^{k+i+\nu}}{1-z^{\nu}}\Biggr)
    \Biggl(\prod_{1\le \nu\le j}\frac{1-z^{k-j+\nu}}{1-z^{\nu}}\Biggr).
\]
The random variables $I_n$ are asymptotically normally 
distributed if
\[
    \frac{\sum_{\nu=1}^i((k+\nu)^4-\nu^4)
    +\sum_{\nu=1}^l((k+i+\nu)^4-\nu^4)
    +\sum_{\nu=1}^j((k-j+\nu)^4-\nu^4)}
    {\left(\sum_{\nu=1}^i((k+\nu)^2-\nu^2)
    +\sum_{\nu=1}^l((k+i+\nu)^2-\nu^2)
    +\sum_{\nu=1}^j((k-j+\nu)^2-\nu^2)\right)^2}\to 0,
\]
which is equivalent to
\[
    \frac{ik(k+i)^3+l(k+i)(k+i+l)^3+j(k^4-j^4)}
    {\bigl(ik(k+i)+l(k+i)(k+i+l)+j(k^2-j^2)\bigr)^2}\to 0.
\]
If we assume that the parameters
$i,j,k,l$ are proportionate to some parameter $t$, that is
$i=\tr{\alpha t}, j=\tr{\beta t}, k=\tr{\gamma t}, l=\tr{\delta t}$,
where $\alpha,\beta,\gamma,\delta>0$ and $\alpha+\gamma+\delta=1$,
then the above condition is satisfied and as a consequence
$I_n$ is asymptotically normally distributed as $t\to
\infty$. This fact has been proved in \cite{bohm05a} by the method
of moments.

\paragraph{Rank statistics} Many test statistics based on ranks
lead to explicit generating functions that are of the form 
\eqref{PNz}, and thus the corresponding limit distribution can be 
dealt with by the tools we established. In particular, we have
the following correspondence between test statistics and 
combinatorial structures; see \cite{van-de-wiel99a} for more 
information. 
\begin{center}
\begin{tabular}{|c|c|}\hline
Kendall's $\tau$ & Inversions in permutations\\ \hline
Mann-Whitney test & Gaussian polynomials \\ \hline
Jonckheere-Terpstra test & Mahonian statistics \\ \hline
\end{tabular}
\end{center}	

On the other hand, the Wilcoxon signed rank test 
(see \cite{wilcoxon47a}) leads to the 
probability generating function of the form
\[
    \prod_{1\le j\le n}\frac{1+z^j}2,
\]
which admits a straightforward generalization to 
(see \cite{van-de-wiel99a} for details)
\[
    \prod_{1\le j\le n}\frac{1+z^{a_j}}2,
\]
where the $a_j$'s can be any real numbers. When they are all 
nonnegative integers, we see, by \eqref{cc}, that the 
associated random variables are asymptotically normally 
distributed if and only if
\[
    \frac{a_1^4+\cdots+a_n^4}
    {(a_1^2+\cdots+a_n^2)^2}\to 0,
\]
as $n\to \infty$. In particular, this applies to Wilcoxon's test 
($a_j=j$) and to Policello and Hettmansperger's test ($a_j = 
\min\{2j,n+1\}$; \cite{policello76a}).

\subsection{Tur\'an-Fej\'er polynomials}

The class of polynomials we consider here (see \eqref{Pnkx} below)
is of interest for several reasons. First, they lead to
asymptotically normally distributed random variables but do not have
the finite-product form \eqref{PNz}. Second, they provide
natural examples with non-normal limit laws when the second
parameter varies. Finally, they have a concrete interpretation in
terms of the partitioning cost of some variants of quicksort.

\cite{fejer37a} studied the Ces\`aro summation of the geometric
series defined by
\[
    F_{n,k}(z) := \sum_{0\le j\le n} F_{j,k-1}(z) \qquad(k\ge1),
\]
with
\[
    F_{n,0}(z) := \sum_{0\le j\le n} z^j,
\]
and \cite{turan49a} proved that all $F_{n,k}^{(k)}(z)$ are
root-unitary for $0\le k\le n$. We characterize all possible limit
laws for the random variables defined via the coefficients of
$F_{n,k}^{(k)}(z)$ for $0\le k\le n$.

By the relation
\[
    F_{n,k}(z) = [w^n]\frac{1}{(1-w)^{k+1}(1-zw)},
\]
where $[w^n]f(w)$ denotes the coefficient of $w^n$ in the Taylor
expansion of $f(w)$, we have
\begin{align*}
    F_{n,k}^{(k)}(z) &= [w^n]\frac{1}{(1-w)^{k+1}}
    \cdot \frac{k! w^k}{(1-zw)^{k+1}}\\
    &= k!\sum_{0\le j\le n-k} \binom{j+k}{k}\binom{n-j}{k} z^j.
\end{align*}
Normalizing this polynomial, we obtain
\begin{align} \label{Pnkx}
    P_{n,k}(z) := \sum_{0\le j\le n-k}
    \frac{\binom{j+k}{k}\binom{n-j}{k}}{\binom{n+k+1}{2k+1}}\,
    z^j,
\end{align}
which gives rise to a sequence of probability generating functions 
of random variables, say $Z_{n,k}$. Note that
\[
    z^kP_{n-k-1,k}(z) = \sum_{k\le j\le n-k-1}
    \frac{\binom{j}{k}\binom{n-1-j}{k}}{\binom{n}{2k+1}}\,
    z^j,
\]
which arises in the analysis of quicksort using the median of $2k+1$
elements; see \cite{sedgewick80a,chern02a} or Appendix.

\begin{lem} For $m\ge0$
\begin{align} \label{Znk-mm}
    \mathbb{E}(Z_{n,k}^m) = \sum_{0\le \ell \le m}S(m,\ell)\ell!
    \frac{\binom{k+\ell}{k}\binom{n+k+1}{2k+\ell+1}}
    {\binom{n+k+1}{2k+1}},
\end{align}
where $S(m,\ell)$ denotes the Stirling numbers of the second kind.
In particular,
\begin{align}\label{Znk-mv}
    \mathbb{E}(Z_{n,k}) = \frac{n-k}{2}\quad\text{and}\quad
    \mathbb{V}(Z_{n,k}) = \frac{(n-k)(n+k+2)}{4(2k+3)}.
\end{align}
\end{lem}
\begin{proof}
By \eqref{Pnkx}, the relation
\[
    j^m = \sum_{0\le \ell \le m}S(m,\ell)j\cdots(j-\ell+1),
\]
and the combinatorial identity
\[
    \sum_{0\le j\le n-k}\binom{j+k}{k}\binom{n-j}{k}\binom{j}{\ell}
    = \binom{k+\ell}{k}\binom{n+k+1}{2k+\ell+1},
\]
(easily proved by convolution), we deduce \eqref{Znk-mm}.
\end{proof}

\begin{thm} The random variables $Z_{n,k}$ are asymptotically
normally distributed if and only if both $k$ and $n-k$ tend to
infinity. If $0\le k=O(1)$, then the limit law is a Beta 
distribution
\begin{align}\label{Znk-kk}
    \frac{Z_{n,k}}{n} \stackrel{d}{\longrightarrow}
    \text{\emph{Beta}}(k,k).
\end{align}
If $1\le \ell :=
n-k=O(1)$, then the limit law is a binomial distribution
\[
    Z_{n,k} \stackrel{d}{\longrightarrow}
    \text{\emph{Binom}}(\ell;\tfrac12).
\]
\end{thm}
\begin{proof}
By \eqref{Znk-mv}, the variance tends to infinity if and only if
$n-k\to\infty$ ($0\le k\le n$). Also we obtain, by \eqref{Znk-mm},
\[
    \frac{\mathbb{E}\left(Z_{n,k}-\frac{n-k}{2}\right)^4}
    {\mathbb{V}(Z_{n,k})^2}-3
    =-\frac{2(3n^2+6n+k^2+4k+6)}{(n-k)(n+k+2)(2k+5)}
    =O\left(\frac{n}{k(n-k)}\right).
\]
The asymptotic normality then follows. We can indeed obtain a local
limit theorem by straightforward calculations from \eqref{Pnkx}.

When $k=O(1)$, we have, by \eqref{Pnkx},
\[
    \frac{\mathbb{E}(Z_{n,k}^m)}{n^m}
    \to \frac{(k+m)!(2k+1)!}{k!(2k+m+1)!}\qquad(m\ge0),
\]
implying that the moment generating function of the limit law
satisfies
\[
    \mathbb{E}(e^{Z_ks})
    =\frac{(2k+1)!}{k!}\sum_{m\ge0}\frac{(k+m)!}{m!(2k+m+1)!}\, s^m
    = \frac{(2k+1)!}{k!k!}\int_0^1 x^k(1-x)^k e^{xs}\dd x,
\]
a Beta distribution. Note that we can express the moment generating
function in terms of Bessel functions as
\begin{align}\label{Zk-Bessel}
    \mathbb{E}(e^{(Z_k-1/2) s})
    &= \left(\frac{is}{4}\right)^{-k-1/2}
    \Gamma(k+\tfrac32) J_{k+\frac12}(is/2) \nonumber\\
    &= \prod_{j\ge1}\left(1+\frac{s^2}
    {4\zeta_{k+1/2,j}^2}\right),
\end{align}
where $J_\alpha$ denotes the Bessel function and the
$\zeta_{\alpha,j}$'s denote the positive zeros of $J_{\alpha}(z)$
arranged in increasing order. By considering
$2(Z_k-1/2)\sqrt{2k+3}$, we obtain \eqref{Exs-ip} with $q_j =
2(2k+3)/\zeta_{k+1/2,j}^2$.

On the other hand, when $\ell := n-k = O(1)$, we have, by
\eqref{Znk-mm},
\[
    P_{n,k}(z)
    \to \left(\frac{1+z}2\right)^\ell,
\]
a binomial distribution. Note that we have the factorization
\[
    \mathbb{E}(e^{(X-\ell/2)s/\sqrt{\ell/4}})
    = \prod_{j\ge1}\left(1+\frac{4s^2}{(2j-1)^2\pi^2\ell}
    \right)^\ell.
\]
\end{proof}

\section{Applications II. Non-normal limit laws}
\label{sec-app-ii}

In addition to the extremal cases of the Tur\'an-Fej\'er
polynomials, we consider in this section more root-unitary
polynomials whose coefficients have a limit distribution that is not
Gaussian. A class of polynomials exhibiting a similar non-Gaussian
behavior is included in Appendix because the proof that they are 
root-unitary is still missing. 

\subsection{Reimer's polynomials}

In the course of investigating the remainder theory of finite
difference, \cite{reimer69a} proved, as a side result, that the
polynomials
\[
    R_{n,m}(y) := \sum_{0\le j\le n} \binom{n}{j}
    y^j \int_j^{j+1} |t(t-1)\cdots(t-n-1)|^m \dd{t}.
\]
have only unit roots. We consider the distribution of the
coefficients of $R_{n,m}(y)$.

For simplicity, we consider only $m=1$ and write $R_{n}=R_{n,1}$.
Define the random variables $X_n$ by
\[
    \mathbb{E}(y^{X_n}) := \frac{R_{n}(y)}{R_{n}(1)}.
\] 
Let 
\[
    A_k := [z^k]\frac z{\log(1-z)}\qquad(k\ge0).
\]
These numbers are (up to sign) known under the name of Cauchy 
numbers; see \cite[pp.\ 293--294]{comtet74a}. See also the recent 
paper \cite{kowalenko10a} for a detailed study of these numbers. 
\begin{lem} For $n\ge1$
\begin{align} \label{reimer-pgf}
    \mathbb{E}(y^{X_n}) = 12\sum_{0\le j\le n}\binom{n}{j}
    y^{n-j}(1-y)^{j}A_{j+2}.
\end{align}
\end{lem}
\begin{proof} We have 
\begin{align}
    R_n(y) &= \sum_{0\le j\le n} \binom{n}{j}(-1)^{n+1+j}
    y^j \int_j^{j+1} t(t-1)\cdots(t-n-1) \dd{t}\nonumber \\
    &= (n+2)!(-1)^{n+1}\sum_{0\le j\le n} \binom{n}{j}(-1)^{j}
    y^j \int_j^{j+1} \binom{t}{n+2} \dd{t}\nonumber\\
    &= (n+2)!(-1)^{n+1}[z^{n+2}]\sum_{0\le j\le n}
    \binom{n}{j}(-1)^{j}
    y^j \int_j^{j+1} (1+z)^t \dd{t}\nonumber\\
    &=(n+2)!(-1)^{n+1}[z^{n+1}]\sum_{0\le j\le n}
    \binom{n}{j}(-1)^{j}
    y^j \frac{(1+z)^j}{\log(1+z)} \nonumber\\
    &= (n+2)![z^{n+1}]\frac{(1-(1-z)y)^n}{\log(1-z)}.
    \label{R-exact}
\end{align}
In particular
\[
    R_{n}(1) = (n+2)! [z]\frac1{\log(1-z)} 
    = \frac{(n+2)!}{12},
\]
and \eqref{reimer-pgf} follows. 
\end{proof}
Note that $A_0 = -1$ and
\[
    A_k = -\sum_{0\le j<k}\frac{A_j}{k+1-j}\qquad(k\ge1).
\]
All $A_k$'s are positive except $A_0$. 

\begin{lem} The moments of $X_n$ satisfy 
\begin{align} \label{R-EXnm}
    \mathbb{E}(X_n^m) = \sum_{0\le k\le m}\tilde{A}_kS(m,k)
    n(n-1)\cdots (n-k+1)\qquad(m\ge0),
\end{align}
where
\begin{align}\label{Bm}
    \tilde{A}_k := 12\sum_{0\le \ell \le k}
    \binom{k}{\ell}(-1)^\ell A_{\ell+2}.
\end{align}
In particular,
\[
    \mathbb{E}(X_n) =\frac n2,\quad
    \mathbb{V}(X_n) = \frac n{60}(4n+11).
\]
\end{lem}
\begin{proof}
By taking $m$-th derivative with respect to $y$ and then
substituting $y=1$ in \eqref{R-exact}, we obtain
\begin{align*}
    \mathbb{E}(X_n(X_n-1)\cdots(X_n-m+1))
    &= 12 [z^{n+1}] \frac{\partial^m}{\partial y^m}
    \frac{(1-(1-z)y)^{n}}{\log(1-z)}\Biggl|_{y=1}\\
    &= 12 [z^{m+1}]\frac{(z-1)^m}{\log(1-z)} \,n(n-1)\cdots(n-m+1),
\end{align*}
which yields \eqref{R-EXnm} since
\[
    \mathbb{E}(X_n^m) = \sum_{0\le k\le m}S(m,k)
    \mathbb{E}(X_{n}(X_{n}-1)\cdots(X_{n}-k+1)).
\]
\end{proof}

\begin{thm} The sequence of random variables $\{X_n/n\}$ converges in
distribution to $X$ whose $m$-th moment equals $\tilde{A}_m$ (defined in
\eqref{Bm}).
\end{thm}
\begin{proof}
By \eqref{R-EXnm}, $\mathbb{E}(X_n^m)\sim \tilde{A}_mn^m$.
Since $A_k = O(1/k)$, we see that $\tilde{A}_m = O(2^m)$,
implying that such a moment sequence determines uniquely a
distribution.
\end{proof}
The limit law has the moment generating function
\begin{align*}
    \mathbb{E}(e^{X s}) &= 12\sum_{m\ge0}\frac{s^m}{m!}
    \sum_{0\le j\le m}\binom{m}{j} (-1)^{j} A_{j+2} \\
    &= 12e^s\sum_{j\ge 0}
    \frac{A_{j+2}}{j!}\, (-s)^j \\
    &= -\frac{12}{2\pi i s}\int_{\mathscr{H}}
    e^{t+s}\left(\frac1{\log(1+\frac st)}
    -\frac ts-\frac12\right)\dd t ,
\end{align*}
where the integration $\int_{\mathscr{H}}$ is taken along some
Hankel contour; see \cite[p.\ 745]{flajolet09a}.

When $m\ge2$, we can apply the same arguments but the 
technicalities are more involved.

\subsection{Chung-Feller's arcsine law}
The classical Chung-Feller theorem states that the number of
positive terms $W_n$ of the sums $S_n = X_1+\cdots+X_n$, where $X_i$
takes $\pm1$ with probability $1/2$ each, has the probability
\[
    \mathbb{P}(W_n=k) = \binom{2k}{k}\binom{2n-2k}{n-k}4^{-n}
    \qquad(k=0,\dots,n).
\]
The limit distribution is an arcsine law (see \cite[\S
III.4]{feller68a})
\[
    \frac{W_n}{n} \stackrel{d}{\longrightarrow} W,
    \quad\text{where}\quad
    \mathbb{P}(W<x) = \frac2\pi\arcsin\sqrt{x}.
\]

The corresponding probability generating function is a polynomial
with only unit roots. Indeed, following the same proof as in
\cite{turan49a}, we can show that $\mathbb{E}(z^{W_n})$ is connected
to Legendre polynomials by the relation
\begin{align*}
    \mathbb{E}(z^{W_n}) &= [v^n] \frac1{\sqrt{(1-v)(1-zv)}}\\
    &= z^{n/2} \text{Legendre}_n
    \left(\frac{z^{1/2}+z^{-1/2}}{2}\right),
\end{align*}
so that the root-unitarity of the left-hand side follows from the
property that Legendre polynomials have only real roots over the
interval $[-1,1]$. Note that the moment generating function of the
arcsin law with zero mean and unit variance is given by the Bessel
function
\begin{align*}
    \mathbb{E}(e^{(W-1/2)s/\sqrt{2}})
    &= e^{-\sqrt{2}s}\left(1+\sum_{k\ge1}\binom{2k}k
    \frac{(s/\sqrt{2})^k} {k!}   \right)\\
    &=J_0(\sqrt{2}\,is) = \prod_{j\ge1}\left(
    1+\frac{2s^2}{\zeta_{0,j}^2}\right),
\end{align*}
where the $\zeta_{0,j}$'s are the positive zeros of $J_0(z)$. So we
have \eqref{Exs-ip} with $q=0$ and $q_j = 4\zeta_{0,j}^{-2}$.

In a more general manner, from the Gegenbauer polynomials, one can
also define the random variables $W_n$ by
\begin{align*}
    \mathbb{E}(z^{W_n}) &= \frac1{\binom{2\alpha+n-1}n}
    [v^n] \frac1{(1-v)^{-\alpha}(1-zv)^{-\alpha}}\\
    &= \sum_{0\le j\le n} \frac{\binom{\alpha+j-1}{j}
    \binom{\alpha+n-j-1}{n-j}}{\binom{2\alpha+n-1}{n}}\, z^j
    \qquad(\alpha>0),
\end{align*}
for which all coefficients are positive and $\mathbb{E}(z^{W_n})$ has
only unit roots. The limit law $W_\alpha$ can be derived as in the
bounded case of the Tur\'an-Fej\'er polynomials
\begin{align*}
    \mathbb{E}(e^{(W_\alpha-1/2) s})
    &= \left(\frac{is}{4}\right)^{-\alpha+1/2}
    \Gamma(\alpha+1/2) J_{\alpha-1/2}(is/2)\\
    &= \prod_{j\ge1}\left(1+\frac{s^2}
    {4\zeta_{\alpha+1/2,j}^2}\right).
\end{align*}
Note that the random variable $2\sqrt{2\alpha+1}(W_\alpha-1/2)$ has
variance one.

For other potential examples, see \cite[Chapter 6]{johnson05a}.

\subsection{Uniform distribution}

The literature abounds with criteria for the root-unitarity of 
polynomials. Among these, \cite{lakatos02a} proved that a complex 
polynomial $P(z) :=\sum_{0\le k\le n}a_k z^k$ with $a_k=a_{n-k}$ 
is root-unitary if 
\[
    |a_n|\ge \sum_{0\le j\le n}|a_n-a_j| ; 
\]
see also \cite{schinzel05a}. In particular, if the coefficients of 
$P(z)$ are close to a constant, then all its roots lie on the unit 
circle. For example, let $E_j = j![z^j]1/\cosh(z)$ denote Euler's 
numbers; then the polynomial 
\[
    P_n(z)=(-1)^n\sum_{0\le j\le n} 
    \binom{2n}{2j} E_{2j}E_{2n-2j}z^j
	= [w^n] \frac{1}{\cos(\sqrt{w})\cos(\sqrt{wz})},
\]
is root-unitary (see \cite{lalin11a}) with non-negative coefficients. 
See also \cite{lalin12a} for more information and other root-unitary
polynomials. Observe that  
\[
	\frac{(-1)^n}{(2n)!}\binom{2n}{2j}E_{2j}E_{2n-2j}
	\sim \frac{4^{n+2}}{\pi^{2n+2}},
\]
as $j, n-j\to\infty$. Thus we can show that the random variables 
associated with the coefficients of $P_n(z)$ will be close to uniform, 
and the limit law is also uniform. Details are omitted here. 

\bibliographystyle{apalike}
\bibliography{polynomials_bibliography}

\section*{Appendix. A class of mixtures of hypergeometric distributions}

Yet another class of polynomials with a similar nature to those of
Tur\'an-Fej\'er arises from the analysis of the partitioning stages
of quicksort and defined as follows (see \cite{chern02a}). Consider
the random variable $Y_n$ defined by
\[
    \mathbb{P}(Y_n=k) = \sum_{0\le j<r}p_j
    \frac{\binom{k}{j}\binom{n-1-k}{r-1-j}}{\binom{n}{r}},
\]
where $r\ge1$ and $\sum_{0\le j<r} p_j =1$ is a given known
distribution. Many concrete examples are discussed in
\cite{chern02a}. Let $P(z) := \sum_{0\le j<r} p_j z^j$. Assume
throughout that $p_j = p_{r-1-j}$ for $1\le j<r$. Numerical evidence
suggested that the probability generating function
\[
    \mathbb{E}(z^{Y_n})
    =\sum_{0\le j<r}p_j \sum_{k} \frac{\binom{k}{j}
    \binom{n-1-k}{r-1-j}}{\binom{n}{r}}z^k
\]
are root-unitary for many natural choices of $\{p_j\}$, but it is
unclear for which class of polynomials $P(z)$ will the polynomials
be root-unitary\footnote{The problem can be formulated by asking for
which class of polynomials $P(z)= \sum_{0\le j<r} p_j z^j$ will the
polynomials
\[
    [w^n]\frac{1}{(1-w)(1-zw)}P\left(
    \frac{1}{(1-w)(1-zw)}\right)
\]
have only unit roots?}. The results below do not depend on the
root-unitarity of $\mathbb{E}(z^{Y_n})$.

Assuming from now on $p_j=p_{r-1-j}$ for $0\le j<r$, we examine the
moment structure of $Y_n$. Note that this assumption implies that
$\sum_{0\le j<r} jp_j = (r-1)/2$.

\begin{lem} The $m$-th moment of $Y_n$ is given explicitly by
\begin{align}\label{EXnm-S}
    \mathbb{E}(Y_n^m)
    = \sum_{0\le h\le m} \nu_{m,h}
    \frac{\binom{n+h}{r+h}}{\binom{n}{r}}.
\end{align}
Here
\[
    \nu_{m,h} := (-1)^{m+h}S(m+1,h+1) \sum_{0\le \ell\le h}
    s(h+1,\ell+1) \pi_\ell,
\]
where the $s(m,h)$ denote the signless Stirling numbers of the first
kind, and $\pi_\ell := \sum_{0\le j<r} p_jj^\ell$. In particular,
$\mathbb{E}(Y_n) = (n-1)/2$ and the variance satisfies
\[
    \mathbb{V}(Y_n) = \frac{(4\pi_2-r^2+3r)n^2+2(6\pi_2-2r^2+3r-1)n
    +8\pi_2-3r^2+3r-2}{4(r+1)(r+2)}.
\]
\end{lem}
Note that $\pi_2-4r^2+3r$ is always positive because
\[
    \frac{r(r-3)}{4}<\pi_1^2 = \frac{(r-1)^2}4 \le \pi_2.
\]
Thus the limit law of $Y_n$ is never Gaussian for finite $r$.
\begin{proof}
By definition,
\begin{align} \label{EXnm}
    \mathbb{E}(Y_n^m) = \sum_{0\le j<r}p_j \sum_{k}k^m
    \frac{\binom{k}{j}\binom{n-1-k}{r-1-j}}{\binom{n}{r}}.
\end{align}
The, using the relation
\[
    k^m = \sum_{0\le h\le m} (-1)^{m+h}S(m+1,h+1)
    (k+1)\cdots(k+m),
\]
we obtain
\begin{align*}
    &\sum_{k}k^m \binom{k}{j}\binom{n-1-k}{r-1-j} \\
    &\qquad= \sum_{0\le h\le m} (-1)^{m+h}S(m+1,h+1) \sum_{k}
    (k+1)\cdots(k+h) \binom{k}{j}\binom{n-1-k}{r-1-j} \\
    &\qquad= \sum_{0\le h\le m} (-1)^{m+h}S(m+1,h+1)
    (j+1)\cdots(j+h)
    \sum_{k}\binom{k+h}{j+h}\binom{n-1-k}{r-1-j}\\
    &\qquad= \sum_{0\le h\le m} (-1)^{m+h}S(m+1,h+1)
    (j+1)\cdots(j+h) \binom{n+h}{r+h}.
\end{align*}
Now, by substituting the expression
\[
    (j+1)\cdots (j+h)
    = \sum_{0\le \ell\le h} s(h+1,\ell+1) j^\ell,
\]
We then obtain \eqref{EXnm-S}.
\end{proof}

\begin{thm} The sequence of random variables $\{Y_n/n\}$ converges in
distribution to a limit law $Y$ whose moment generating function
satisfies
\begin{align} \label{eYs}
    \mathbb{E}(e^{Ys}) =
    r\sum_{0\le j<r}p_j \binom{r-1}{j} \int_0^1 e^{xs}
    x^{r-1-j}(1-x)^j \rm{d} x.
\end{align}
\end{thm}
\begin{proof}
Indeed, by \eqref{EXnm-S},
\[
    \frac{\mathbb{E}(Y_n^m)}{n^m}
    \sim \frac{r!}{(r+m)!}\sum_{0\le j<r} p_j
    (j+1)\cdots(j+m),
\]
so that
\[
    \frac{\mathbb{E}(Y_n^m)}{n^m}
    \stackrel{d}{\longrightarrow} Y,
\]
where the moment generating function of $X$ is given by
\begin{align*}
    \mathbb{E}(e^{Ys})
    &= \sum_{m\ge0}\frac{s^m}{m!} \sum_{0\le j<r} p_j
    \frac{\Gamma(j+m+1)\Gamma(r+1)}{\Gamma(r+m+1)\Gamma(j+1)}\\
    &= \sum_{m\ge0}\frac{s^m}{m!} \sum_{0\le j<r} p_j
    \frac{\Gamma(r-j+m)\Gamma(r+1)}{\Gamma(r+m+1)\Gamma(r-j)}\\
    &= r\sum_{m\ge0}\frac{s^m}{m!} \sum_{0\le j<r} p_j
    \binom{r-1}{j} \int_0^1 (1-x)^j x^{r-1-j+m} \text{d}x,
\end{align*}
which proves \eqref{eYs}. The justification of the unique
characterization of this limit law is straightforward.
\end{proof}

Examples.
\begin{itemize}
\item The median of $(2k+1)$ elements: This corresponds to the
case when $r=2k+1$ and $p_k=1$. We then obtain
\[
    \mathbb{E}(e^{Ys}) = \frac{(2k+1)!}{k!k!}
    \int_0^1 e^{xs} x^k(1-x)^k \text{d} x,
\]
a Beta distribution (with a Bessel-type infinite-product
representation); see \eqref{Znk-kk} and \eqref{Zk-Bessel}.

\item Uniform distributon: In this case, $p_j=1/r$, $1\le j<r$. 
This is algorithmically uninteresting, but has the limit moment 
generating function $(e^s-1)/s$.

\item The ninther (the median of three medians, each being the 
median of three elements): This is the case when
\[
    \{p_j\}_{j=0,\dots,8} = \left\{0,0,0,\tfrac 3{14}, \tfrac47,
    \tfrac{3}{14},0,0,0\right\}.
\]
We have a mixture of Beta distributions for the limit law.

\end{itemize}
Many sophisticated cases can be found in \cite{chern02a}.
\end{document}